\documentclass[11pt]{article}
\usepackage[colorlinks=false]{hyperref}
\usepackage{amsfonts, amsmath, amssymb, amsthm, gensymb}
\usepackage{subcaption, enumitem, tikz, bm}
\definecolor{darkgray}{RGB}{64,64,64}
\definecolor{litegray}{RGB}{192,192,192}

\usetikzlibrary{decorations.pathreplacing, calc}
\tikzstyle{vertex}=[circle, draw, fill=litegray, inner sep=0pt, minimum width=4pt]
\tikzstyle{braket}=[decorate,decoration={brace,amplitude=10pt},xshift=0pt,yshift=-10pt, darkgray]

\author{Zilin Jiang\thanks{Department of Mathematics, Massachusetts Institute of Technology, Cambridge, MA 02139, USA. Email: {\tt zilinj@mit.edu}. Supported in part by Israel Science Foundation (ISF) grant nos 1162/15, 936/16.}${\ }^{,}$\footnotemark[3] \and Alexandr Polyanskii\thanks{Moscow Institute of Physics and Technology, Institute for Information Transmission Problems RAS, and Caucasus Mathematical Center, Adyghe State University. Email: {\tt alexander.polyanskii@yandex.ru}. Supported in part by ISF grant no. 409/16, by the Russian Foundation for Basic Research through grant no. 15-01-03530 A, and by the Leading Scientific Schools of Russia through grant no. NSh-6760.2018.1.}${\ }^{,}$\footnote{The work was done when the authors were postdoctoral fellows at Technion -- Israel Institute of Technology.}}

\title{Forbidden subgraphs for graphs of bounded spectral radius,\\ with applications to equiangular lines}

\date{}

\newtheorem{theorem}{Theorem}
\newtheorem{corollary}[theorem]{Corollary}
\newtheorem{conjecture}{Conjecture}

\newtheorem{lemma}[theorem]{Lemma}
\newtheorem{proposition}[theorem]{Proposition}

\theoremstyle{definition}
\newtheorem{definition}{Definition}

\theoremstyle{remark}
\newtheorem{remark}{Remark}

\newcommand{\dset}[2]{\left\{#1 : #2\right\}}
\newcommand{\sset}[1]{\left\{#1\right\}}

\newcommand{\al}{\alpha}
\newcommand{\la}{\lambda}
\newcommand{\eps}{\varepsilon}
\newcommand{\downto}{\searrow}
\newcommand{\upto}{\nearrow}
\newcommand{\from}{\colon}
\newcommand{\N}{\mathbb{N}}
\newcommand{\Z}{\mathbb{Z}}
\newcommand{\Q}{\mathbb{Q}}
\newcommand{\R}{\mathbb{R}}
\newcommand{\F}{\mathcal{F}}
\newcommand{\G}{\mathcal{G}}
\newcommand{\floor}[1]{\left\lfloor#1\right\rfloor}
\newcommand{\abs}[1]{\left\lvert#1\right\rvert}
\newcommand{\norm}[1]{\left\lVert#1\right\rVert}
\newcommand{\normsq}[1]{\norm{#1}^2}
\newcommand\ip[2]{\left\langle#1,#2\right\rangle}
\newcommand{\tr}[1]{\operatorname{tr}\left({#1}\right)}
\newcommand{\rank}[1]{\operatorname{rank}\left({#1}\right)}
\newcommand{\E}[1]{\operatorname{E}\left[{#1}\right]}
\newcommand{\pr}[1]{\operatorname{Pr}\left({#1}\right)}

\begin{document}

\maketitle

\begin{abstract}
  The spectral radius of a graph is the largest eigenvalue of its adjacency matrix. Let $\mathcal{F}(\lambda)$ be the family of connected graphs of spectral radius $\le \lambda$. We show that $\mathcal{F}(\lambda)$ can be defined by a finite set of forbidden subgraphs if and only if $\lambda < \lambda^* := \sqrt{2+\sqrt{5}} \approx 2.058$ and $\lambda \not\in \{\alpha_2, \alpha_3, \dots\}$, where $\alpha_m = \beta_m^{1/2} + \beta_m^{-1/2}$ and $\beta_m$ is the largest root of $x^{m+1}=1+x+\dots+x^{m-1}$. The study of forbidden subgraphs characterization for $\mathcal{F}(\lambda)$ is motivated by the problem of estimating the maximum cardinality of equiangular lines in the $n$-dimensional Euclidean space $\mathbb{R}^n$ --- a family of lines through the origin such that the angle between any pair of them is the same. Denote by $N_\alpha(n)$ the maximum number of equiangular lines in $\mathbb{R}^n$ with angle $\arccos\alpha$. We establish the asymptotic formula $N_\alpha(n) = c_\alpha n + O_\alpha(1)$ for every $\alpha \ge \frac{1}{1+2\lambda^*}$. In particular, $N_{1/3}(n) = 2n+O(1)$ and $N_{1/5}(n), N_{1/(1+2\sqrt{2})}(n) = \frac{3}{2}n+O(1)$. Besides we show that $N_\alpha(n) \le 1.49n + O_\alpha(1)$ for every $\alpha \neq \tfrac{1}{3}, \tfrac{1}{5}, \tfrac{1}{1+2\sqrt{2}}$, which improves a recent result of Balla, Dr\"axler, Keevash and Sudakov.
\end{abstract}

\section{Introduction} \label{intro}

The \emph{spectral radius} of a graph $G$, denoted by $\la_1(G)$, is the largest eigenvalue of its adjacency matrix. Let $\F(\la)$ be the family of connected graphs of spectral radius $\le \la$. It is well known that $\la_1$ is monotone in the sense that $\la_1(G_1) \le \la_1(G_2)$ if $G_1$ is a subgraph of $G_2$, moreover $\la_1(G_1) < \la_1(G_2)$ if $G_1$ is a proper subgraph of a connected graph $G_2$. This implies that $\F(\la)$ is closed under taking connected subgraphs. It is natural to ask if $\F(\la)$ can be defined by a finite set of forbidden subgraphs. We determine the set of $\la$ for which the answer is yes.

\begin{theorem} \label{forb_main}
  For $m = 1, 2, \dots$, let $\beta_m$ be the largest root of $x^{m+1} = 1 + x + \dots + x^{m-1}$, and let $\al_m := \beta_m^{1/2} + \beta_m^{-1/2}$. For every $\la < \la^* := \sqrt{2+\sqrt{5}}$ such that $\la \not\in \sset{\al_2, \al_3, \dots}$, there exist finitely many graphs $G_1, G_2, \dots, G_n$ such that $\F(\la)$ consists exactly of the connected graphs which do not contain any of $G_1, G_2, \dots, G_n$ as a subgraph. However, the same conclusion does not hold for any $\la \in \sset{\al_2,\al_3,\dots}\cup [\la^*,\infty)$.
\end{theorem}

The mysterious number $\la^*$ has appeared before in the study of $\F(\la)$. After Smith~\cite{MR0266799} precisely determined the graphs in $\F(2)$, Cvetkovi\'c, Doob and Gutman~\cite{MR683990} classified the graphs in $\F(\la^*)$, which were later completely described by Brouwer and Neumaier~\cite{MR986880}. On the face of it, for $\la < \la^*$, Theorem~\ref{forb_main} may appear to be an immediate consequence of Brouwer and Neumaier's theorem. However, Theorem~\ref{forb_main} is genuinely about the minimal graphs outside $\F(\la)$.

Our motivation to understand the forbidden subgraphs characterization for $\F(\la)$ comes from the problem of estimating the maximum cardinality $N(n)$ of \emph{equiangular lines} in the $n$-dimensional Euclidean space $\R^n$ --- a family of lines through the origin such that the angle between any pair of them is the same. It is considered to be one of the founding problems of algebraic graph theory to determine $N(n)$. The ``absolute bound'' $N(n) \le {{n+1}\choose 2}$ was established by Gerzon (see \cite[Theorem~3.5]{MR0307969}). For a long time, it was an open problem to determine whether $n^2$ is the correct order of magnitude until a remarkable construction of de Caen~\cite{MR1795615} shows that $N(n) \ge \frac{2}{9}(n+1)^2$ for $n$ of the form $n = 6\cdot 4^k-1$ (see \cite{MR3423477} for a generalization and \cite{MR3345298} for an alternative constructions). In these constructions, the common angle tends to $\pi/2$ as dimension grows.

The question of determining the maximum number $N_\al(n)$ of equiangular lines in $\R^n$ with common angle $\arccos\al$ was first raised by Lemmens and Seidel~\cite{MR0307969} in 1973, who showed that $N_{1/3}(n) = 2n-2$ for $n \ge 15$ and also conjectured that $N_{1/5}(n)$ equals $\lfloor{3(n-1)/2}\rfloor$ for $n$ sufficiently large. This conjecture was later confirmed by Neumaier~\cite{MR986870} (see also \cite{MR3423477}).

In the absence of exact formulas for $N_{\al}(n)$ except when $\al = 1/3,1/5$, various upper bounds were established. The ``relative bound'' $N_\al(n) \le \frac{1-\al^2}{1-n\al^2}\cdot n$ is valid in small dimensions $n < 1/\al^2$ (see \cite[Lemma~6.1]{MR0200799} and \cite[Theorem~3.6]{MR0307969}), and a general bound attributed to Neumann~\cite[Theorem~3.4]{MR0307969} states that $N_\al(n)\le 2n$ unless $1/\al$ is an odd integer. For many years, a linear upper bound for $N_\alpha(n)$ was not known when $1/\al$ is an odd integer bigger than $5$. An important progress was recently made by Bukh~\cite{MR3477753}, who proved that  $N_\alpha(n) \le c_\al n$ for some $c_\al = 2^{O(1/\al^2)}$. Subsequently, Balla, Dr\"axler, Keevash and Sudakov~\cite{MR3742757} drastically improved the upper bound to $N_\al(n) \le 1.93 n$ for sufficiently large $n$ relative to $\al$ whenever $\al \neq \tfrac{1}{3}$. A universal upper bound $N_\al(n) \le \left(2/(3\al^2)+4/7\right)n+2$ for all $n\in \N$ when $1/\al$ is an odd integer was later found by Glazyrin and Yu~\cite{MR3787558}.

In this paper, we combine Theorem~\ref{forb_main} and the framework developed in \cite{MR3742757} to establish a result of the form $N_\al(n) = c_\alpha n + O(1)$ for all $\al \ge \frac{1}{1+2\la^*}$. To describe the coefficient $c_\al$, we introduce the following notion.

\begin{definition}
  Given $\la > 0$, the \emph{spectral radius order} $k(\la)$ of $\la$ is the smallest order of a graph with spectral radius $\la$. In case $\la$ is not the spectral radius of any graph, we write $k(\la) = \infty$.
\end{definition}

\begin{theorem} \label{equi_main}
  Given $\al \in (0, 1)$ such that $\la := \frac{1-\al}{2\al} \le \la^* := \sqrt{2 + \sqrt{5}}$, the maximum number $N_\al(n)$ of equiangular lines in $\R^n$ with angle $\arccos\al$ equals $\frac{k(\la)}{k(\la) - 1}\cdot n + O(1)$, where $k(\la)$ is the spectral radius order of $\la$; in case $k(\la) = \infty$, $N_\al(n) = n + O(1)$.
\end{theorem}

\begin{remark}
  Throughout, as all the big-O notations depend on $\alpha$, we suppress the subscript in $O_\al(\cdot)$.
\end{remark}

Applying the above theorem to the connected graphs with $2$ or $3$ vertices,
\medskip
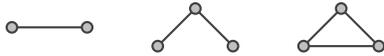
\begin{figure}[h]
  \centering
  \begin{tikzpicture}[thick, scale=0.5, baseline=(v.base)]
    \coordinate (v) at (0,0);
    \draw[darkgray] (0,0) node[vertex]{} -- (2,0) node[vertex]{};
  \end{tikzpicture}\qquad
  \begin{tikzpicture}[thick, scale=0.5, baseline=(v.base)]
    \coordinate (v) at (0,0.5);
    \draw[darkgray] (0,0) node[vertex]{} -- (1,1) node[vertex]{} -- (2,0) node[vertex]{};
  \end{tikzpicture}\qquad
  \begin{tikzpicture}[thick, scale=0.5, baseline=(v.base)]
    \coordinate (v) at (0,0.5);
    \draw[darkgray] (0,0) node[vertex]{} -- (1,1) node[vertex]{} -- (2,0) node[vertex]{} -- cycle;
  \end{tikzpicture}%
\end{figure}\\
whose spectral radii are $1, \sqrt{2}$ and $2$ respectively, we obtain the known results $N_{1/3}(n) = 2n + O(1)$, $N_{1/5}(n) = \frac{3}{2}n + O(1)$ and surprisingly a new result $N_{1/(1+2\sqrt{2})}(n) = \frac{3}{2}n + O(1)$. This refutes the second part of Conjecture 6.1 in \cite{MR3742757}. The first part of the conjecture, which was also raised by Bukh~\cite[Conjecture 8]{MR3477753}, says the following.

\begin{conjecture} \label{conj_odd}
  The maximum number $N_\al(n)$ of equiangular lines in $\R^n$ with angle $\arccos\al$ equals $\frac{1+\al}{1-\al}\cdot n + O(1)$ if $1/\al$ is an odd natural number.
\end{conjecture}

When $1/\al$ is an odd natural number, or equivalently $\la = \frac{1-\al}{2\al}$ is a natural number, the complete graph on $\la + 1$ vertices is the smallest graph with spectral radius $\la$, hence $k(\la) = \la + 1$. In this case, the linear coefficient $\frac{k(\la)}{k(\la)-1}$ in Theorem~\ref{equi_main} matches the linear coefficient $\frac{1+\al}{1-\al}$ in Conjecture~\ref{conj_odd}. This phenomenon motivates the following stronger conjecture.

\begin{conjecture} \label{main_conj}
  The maximum number $N_\al(n)$ of equiangular lines in $\R^n$ with angle $\arccos\al$ equals $\frac{k(\la)}{k(\la)-1}\cdot n + O(1)$, where $k(\la)$ is the spectral radius order of $\la := \frac{1-\al}{2\al}$; in case $k(\la) = \infty$, $N_\al(n)=n+O(1)$.
\end{conjecture}

One of the unknown cases that are not addressed by Theorem~\ref{equi_main} is $\al = 1/7$. Conjecture~\ref{conj_odd} asserts that $N_{1/7}(n) = \frac{4}{3}n+O(1)$. For such cases, we derive an asymptotic upper bound on $N_{\al}(n)$ (see Theorem~\ref{upper_bound}), from which it follows that $N_{1/7}(n) \le \left(\frac{4}{3}+\frac{1}{36}+o(1)\right)n$, and $N_\al(n) \le 1.49n+O(1)$ for $\al \neq 1/3, 1/5, 1/(1+2\sqrt{2})$.

The rest of the paper is organized as follows. In Section~\ref{forb}, we prove Theorem~\ref{forb_main}. In Section~\ref{framework}, we adapt the framework developed by Balla et al. In Section~\ref{3.2}, we apply Theorem~\ref{forb_main} to estimate $N_\al(n)$ when $\la = \frac{1-\al}{2\al} \in (0, \la^*)\setminus\sset{\al_2,\al_3,\dots}$. In Section~\ref{3.3} we complete the proof of Theorem~\ref{equi_main}. In Section~\ref{3.4}, we extrapolate our method to obtain an upper bound on $N_\al(n)$ when $\la > \la^*$. In the concluding section we discuss evidences supporting Conjecture~\ref{main_conj} and a possible extension of our method.

\section{Forbidden subgraphs of $\F(\la)$} \label{forb}

Suppose there is a finite forbidden subgraphs characterization, say $G_1, G_2, \dots, G_n$, for $\F(\la)$. The monotonicity of spectral radius implies that no connected graph has spectral radius in the open interval $(\la, \min\dset{\la_1(G_i)}{i\in[n]})$. Let $\Lambda_1$ consist of the spectral radii of all connected graphs of all orders. The contrapositive of the above observation says the following.

\begin{proposition} \label{right-sided-limit}
  Let $\lim_+\Lambda_1 := \dset{\la \in \R}{(\la, \la + \eps)\cap \Lambda_1 \neq \emptyset \text{ for all }\eps > 0}$ be the set of right-sided limit points of $\Lambda_1$. The family $\F(\la)$ does not have a finite forbidden subgraphs characterization for all $\la \in \lim_+\Lambda_1$.
\end{proposition}

Hoffman was interested in a related set $R$ consisting of the largest eigenvalues of all symmetric matrices of all orders with non-negative integer entries, and he proved the following theorem on the limit points of $R$.

\begin{theorem}[Hoffman~\cite{MR0347860}] \label{hoffman}
  Let $\al_m$ be defined as in Theorem~\ref{forb_main}. Then $2 = \al_1 < \al_2 < \dots$ are all the limit points of $R$ smaller than $\lim_m \al_m = \lambda^*$.
\end{theorem}

In fact, Hoffman proved the above theorem by first showing \cite[Proposition~2.1]{MR0347860} that $\Lambda_1 = R$. He also computed the limit of spectral radii for several families of graphs. We compile some of his computation and other relevant results in the following lemma, the proof of which is presented in Appendix~\ref{app}. We use the notation $\al_n \upto \al$ if $\al_1 < \al_2 < \dots$ and $\lim_n \al_n = \al$, and $\al_n \downto \al$ if $\al_1 > \al_2 > \dots$ and $\lim_n \al_n = \al$.

\begin{lemma} \label{abcdef}
  Let $\al_m$ be defined as in Theorem~\ref{forb_main}. Denote by $C_n$ the cycle with $n$ vertices, $P_n$ the path with $n$ vertices and $S_n$ the star with $n$ leaves. Define $A_n, B_{m_1,n,m_2}, D_n, E_{m,n}, F_n$ as below. The spectral radii of these graphs satisfy:
  (a) $\la_1(A_n)\upto 3/\sqrt{2}$\footnote{It was mistakenly asserted that $\la_1(A_n)\upto 4/\sqrt{3}$ in \cite{MR0347860}. However it will not affect the main result of \cite{MR0347860} as the correct limit $3/\sqrt{2}$ is still $> \la^*$.},
  (b\textsubscript{1}) $\la_1(B_{1,n,1}) = 2$,
  (b\textsubscript{2}) $\la_1(B_{m_1,n,m_2}) \downto \max(\al_{m_1}, \al_{m_2})$ for fixed $(m_1, m_2) \neq (1,1)$,
  (c) $\la_1(C_n) = 2$,
  (d) $\la_1(D_n) \downto \lambda^*$,
  (e) $\la_1(E_{m,n}) \upto \al_m$ for fixed $m$,
  (f) $\la_1(F_n) \upto \la^*$,
  (p) $\la_1(P_n) = 2\cos(\tfrac{\pi}{n+1}) \upto 2$,
  (s) $\la_1(S_n) = \sqrt{n}$.
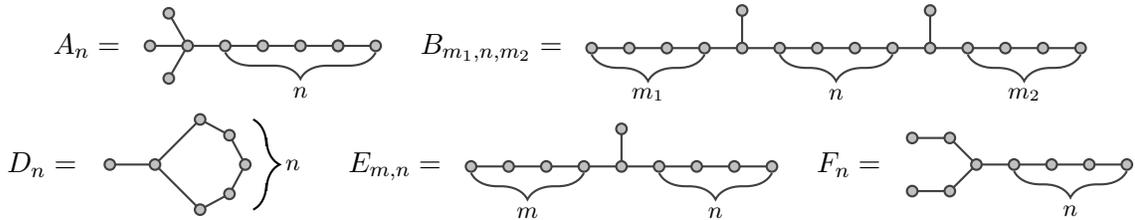
  \begin{figure}[h]
    \centering
    \begin{tikzpicture}[thick, scale=0.25, baseline=(v.base)]
      \draw[darkgray] (-3,0) -- +(120:2) node[vertex]{};
      \draw[darkgray] (-3,0) -- +(-120:2) node[vertex]{};
      \draw[darkgray] (-5,0) node[vertex]{} -- (-3,0);
      \draw[darkgray] (-3,0) node[vertex]{} -- (-1,0);
      \draw[darkgray] (-1,0) node[vertex]{} -- (1,0);
      \draw[darkgray] (1,0) node[vertex]{} -- (3,0);
      \draw[darkgray] (3,0) node[vertex]{} -- (5,0) node[vertex]{};
      \draw[darkgray] (5,0) node[vertex]{} -- (7,0) node[vertex]{};
      \draw (-6, 0) node[left] (v) {$A_n = $};
      \draw [braket] (7,0) -- (-1,0) node [black,midway,yshift=-15pt] {\footnotesize $n$};
    \end{tikzpicture}\quad
    \begin{tikzpicture}[thick, scale=0.25, baseline=(v.base)]
      \draw[darkgray] (3,0) -- +(0,2) node[vertex]{};
      \draw[darkgray] (-5,0) node[vertex]{} -- (-3,0);
      \draw[darkgray] (-3,0) node[vertex]{} -- (-1,0);
      \draw[darkgray] (-1,0) node[vertex]{} -- (1,0);
      \draw[darkgray] (13,0) -- +(0,2) node[vertex]{};
      \draw[darkgray] (1,0) node[vertex]{} -- (3,0) node[vertex]{} -- (5,0) node[vertex]{} -- (7,0) node[vertex]{} -- (9,0) node[vertex]{} -- (11,0) node[vertex]{} -- (13,0) node[vertex]{} -- (15,0) node[vertex]{} -- (17,0) node[vertex]{} -- (19,0) node[vertex]{} -- (21,0) node[vertex]{};
      \draw (-6, 0) node[left] (v) {$B_{m_1,n,m_2} = $};
      \draw [braket] (1,0) -- (-5,0) node [black,midway,yshift=-15pt] {\footnotesize $m_1$};
      \draw [braket] (11,0) -- (5,0) node [black,midway,yshift=-15pt] {\footnotesize $n$};
      \draw [braket] (21,0) -- (15,0) node [black,midway,yshift=-15pt] {\footnotesize $m_2$};
    \end{tikzpicture}\quad
    \begin{tikzpicture}[thick, scale=0.3, baseline=(v.base)]
      \draw[darkgray] (-5,0) node[vertex]{} -- (-3,0) node[vertex]{} -- (-1,2) node[vertex]{} -- (0.3,1.3) node[vertex]{} -- (1,0) node[vertex]{} -- (0.3,-1.3) node[vertex]{} -- (-1,-2) node[vertex]{} -- (-3,0);
      \draw (-6, 0) node[left] (v) {$D_n = $};
      \draw [decorate,decoration={brace,amplitude=10pt},xshift=10pt,yshift=0pt] (1,2) -- (1,-2) node [black,midway,xshift=15pt] {\footnotesize $n$};
    \end{tikzpicture}\quad
    \begin{tikzpicture}[thick, scale=0.25, baseline=(v.base)]
      \draw[darkgray] (-1,0) node[vertex]{} -- (1,0);
      \draw[darkgray] (7,0) -- +(0,2) node[vertex]{};
      \draw[darkgray] (1,0) node[vertex]{} -- (3,0) node[vertex]{} -- (5,0) node[vertex]{} -- (7,0) node[vertex]{} -- (9,0) node[vertex]{} -- (11,0) node[vertex]{} -- (13,0) node[vertex]{} -- (15,0) node[vertex]{};
      \draw (-2, 0) node[left] (v) {$E_{m,n} = $};
      \draw [braket] (5,0) -- (-1,0) node [black,midway,yshift=-15pt] {\footnotesize $m$};
      \draw [braket] (15,0) -- (9,0) node [black,midway,yshift=-15pt] {\footnotesize $n$};
    \end{tikzpicture}\quad
    \begin{tikzpicture}[thick, scale=0.25, baseline=(v.base)]
      \coordinate (a) at (135:2);
      \coordinate (b) at (-135:2);
      \draw[darkgray] (0,0) -- (a) node[vertex]{} -- ($(a)-(2,0)$) node[vertex]{};
      \draw[darkgray] (0,0) -- (b) node[vertex]{} -- ($(b)-(2,0)$) node[vertex]{};
      \draw[darkgray] (0,0) node[vertex]{} -- (2,0) node[vertex]{} -- (4,0) node[vertex]{} -- (6,0) node[vertex]{} -- (8,0) node[vertex]{};
      \draw (-4.5, 0) node[left] (v) {$F_n = $};
      \draw [braket] (8,0) -- (2,0) node [black,midway,yshift=-15pt] {\footnotesize $n$};
    \end{tikzpicture}
  \end{figure}
\end{lemma}

The work of finding all the limit points of $\Lambda_1$ was completed by Shearer.

\begin{theorem}[Shearer~\cite{MR986863}] \label{shearer}
  For any $\la \ge \lambda^*$, there exists a sequence of distinct graphs $G_1, G_2, \dots$ such that $\lim\la_1(G_i) = \la$.
\end{theorem}

Equipped with these facts from spectral graph theory, we are in the position to prove Theorem~\ref{forb_main}.

\begin{proof}[Proof of Theorem~\ref{forb_main}]
  Note that Lemma~\ref{abcdef}(b\textsubscript{2}) and Theorem~\ref{shearer} imply that $\sset{\al_2, \al_3, \dots}\cup[\lambda^*, \infty) \subseteq \lim_+\Lambda_1$. Thus by Proposition~\ref{right-sided-limit}, the second half of Theorem~\ref{forb_main} is proved. For the first half of Theorem~\ref{forb_main}, we break the proof into two cases.

  \noindent\textbf{Case 1: $\la < 2$.} By Lemma~\ref{abcdef}, we have $S_4 \notin \F(\la)$ and $P_n \notin \F(\la)$ for some $n$. Clearly, a connected graph $G$ that contains neither $S_4$ nor $P_n$ has maximum degree $\le 3$ and diameter $< n$. Therefore \[
    \sset{S_4, P_n} \cup \sset{\text{connected graph }G\notin\F(\la)\text{ with maximum degree} \le 3 \text{ and diameter}<n}
  \] is a finite forbidden subgraphs characterization for $\F(\la)$.

  \noindent\textbf{Case 2: $\la \in [2, \la^*)\setminus\sset{\al_2, \al_3, \dots}$.} Choose the smallest $m$ such that $\al_m > \la$. Observe that $\la \in [2, \al_2)$ when $m = 2$, and $\la \in (\al_{m-1},\al_m)$ when $m \ge 3$. Then choose, in view of Lemma~\ref{abcdef}, $n > m$ such that $A_n, E_{m, n}, F_n \notin \F(\la)$ and $B_{m_1, n', m_2} \in \F(\la)$ for all $m_1, m_2 < m$ and $n' > n$. In fact, when $m = 2$, the requirement that $B_{m_1,n',m_2} \in \F(\la)$ for all $m_1,m_2 < m$ and $n' > n$ is redundant because of Lemma~\ref{abcdef}(b\textsubscript{1}). By Lemma~\ref{abcdef}, we also have $S_5 \notin \F(\la)$, $B_{m, 0, 1}, B_{m, 1, 1}, \dots, B_{m, n, 1}\notin \F(\la)$ and $D_2, D_3, \dots, D_{m + n} \notin \F(\la)$.

  We claim that if a connected graph $G$ contains none of
  \[
    \G_0 := \sset{S_5, A_n, B_{m,0,1}, B_{m,1,1}, \dots, B_{m,n,1}, D_2, D_3, \dots, D_{m+n}, E_{m,n}, F_n},
  \]
  then $G$ is a path, or a cycle, or $E_{m_0,n_0}$ for some $m_0 < m$ and some $n_0$, or $B_{m_1, n', m_2}$ for some $m_1, m_2 < m$ and some $n' > n$, or $G$ has maximum degree $\le 4$ and radius $< m+n$. Recall that the eccentricity of a vertex $v$ in $G$ is the maximum distance between $v$ and any other vertex of $G$, and the radius of $G$ is the minimum eccentricity of any vertex in $G$.

  Notice that, by Lemma~\ref{abcdef}, the spectral radii of paths and cycles are at most $2$, the spectral radius of $E_{m_0,n_0}$ is $< \al_{m_0} \le \la$ for all $m_0 < m$ and $n_0$ by the choice of $m$, and the spectral radius of $B_{m_1,n',m_2}$ is $\le \lambda$ for all $m_1, m_2 < m$ and $n' > n$ by the choice of $n$. The claim thus implies that
  \[
    \G_0 \cup \sset{\text{connected graph }G\notin\F(\la)\text{ with maximum degree}\le 4\text{ and radius}< m+n}
  \] is a finite forbidden subgraphs characterization for $\F(\la)$.

  Finally, it suffices to prove the claim. Because $G$ does not contain $D_2, \dots, D_{m+n}, E_{m,n}$ and $D_{m+n+1}, D_{m+n+2}, \dots$ contain $E_{m,n}$, the graph $G$ does not contain $D_2, D_3, \dots$, and so $G$ must be a tree or a cycle. Suppose that $G$ is not a path or a cycle. As $G$ does not contain $S_5$, $G$ is a tree of maximum degree $3$ or $4$. If the degree of a vertex $v$ is $4$, then as $G$ does not contain $A_n$, the eccentricity of $v$ is $< n$ and so the radius of $G$ is $< n$. If a vertex $v$ of degree $3$ is not adjacent to any leaf, then as $G$ does not contain $F_n$, the eccentricity of $v$ is $< n$ and so the radius of $G$ is $< n$.

  Hereafter, we may assume that the maximum degree of $G$ is $3$ and that every vertex of degree $3$ is adjacent to a leaf, in other words, $G$ is a caterpillar\footnote{A caterpillar is a tree in which all the vertices are within distance $1$ of a central path.} of maximum degree $3$. Let $v_1,\dots, v_\ell$ be the central path of $G$. Because $G$ does not contain $B_{m,0,1}, B_{m,1,1},\dots, B_{m,n,1}, E_{m,n}$ and $B_{m,n+1,1}, B_{m,n+2,1},\dots$ contain $E_{m,n}$, the graph $G$ does not contain $B_{m,0,1}, B_{m,1,1}, B_{m,2,1}, \dots$, and so for every pair of vertices $v_i, v_j$ of degree $3$ with $i < j$ the distances between $v_1, v_i$ and between $v_j, v_\ell$ are $< m$. We finish the proof of the claim by checking the following three cases.
  \begin{enumerate}[nosep]
    \item If $G$ has only one vertex of degree $3$, then as $G$ does not contain $E_{m,n}$, the graph $G$ must be $E_{m_0,n_0}$ for some $m_0 < m$ and some $n_0$.
    \item If $G$ has exactly two vertices of degree $3$, then the graph $G$ must be $B_{m_1, n', m_2}$ for some $m_1, m_2 < m$ and some $n'$. In this case, either $n' > n$ or the radius of $G$ is $\le \frac{1}{2}(2m+n) < m+n$.
    \item If $G$ has at least $3$ vertices, say $v_i, v_j, v_k$ with $i < j < k$, of degree $3$, then the distances from $v_j$ to $v_1$ and to $v_\ell$ are $< m$, hence the radius of $G$ is $< m$. \qedhere
  \end{enumerate}
\end{proof}

\section{Equiangular lines} \label{equi}

\subsection{The framework to estimate $N_\al(n)$} \label{framework}

We shall set the stage for the proof of Theorem~\ref{equi_main} by adapting the framework developed in Section~2.1 of \cite{MR3742757} to estimate $N_\al(n)$.

\begin{definition}
  Let $L$ be a subset of the interval $[-1, 1)$. A finite set $C$ of unit vectors in $\R^n$ is called a \emph{spherical $L$-code} if $\ip{v_1}{v_2}\in L$ for any pair of distinct vectors $v_1, v_2$ in $C$. The \emph{Gram matrix} $M$ is given by $M_{ij} = \ip{v_i}{v_j}$. The \emph{underlying graph} $G$ is defined as follows: let $C$ be its vertex set, and for any distinct $v_i, v_j\in C$, we put the edge $(v_i, v_j)$ if and only if $\ip{v_i}{v_j} < 0$.
\end{definition}

By choosing a unit vector in the direction of each line in a family of equiangular lines with angle $\arccos\al$ in $\R^n$, we obtain a spherical $\sset{\pm \al}$-code $C_\al$ in $\R^n$. The key insight of Lemma~2.8 of Balla et al.~\cite{MR3742757} is that one can obtain a spherical $L(\al, t)$-code $C$ in $\R^n$, where
\[
  L(\al, t) := \sset{-\frac{1}{\la}\cdot\left(1-\frac{1}{t+\al^{-1}}\right)+\frac{1}{t+\al^{-1}}, \frac{1}{t+\al^{-1}}}, \la := \frac{1-\al}{2\al} \text{ and } t = \floor{\log\log n},
\]
from a spherical $\sset{\pm \al}$-code $C_\al$ such that $\abs{C_\al} \le \abs{C} + o(n)$, and moreover the size of $C$ is much easier to bound from above than $C_\al$. However, to estimate $N_\al(n)$ up to a constant error relative to $\al$, we need instead $\abs{C_\al} \le \abs{C} + O(1)$, which requires $t$ be independent from $n$. Although Lemma~\ref{balla} is stated differently from Lemma~2.8 of Balla et al.~\cite{MR3742757}, the same proof goes through without alternation. For the convenience of the readers, we include here a streamlined proof.

\begin{lemma} \label{balla}
  Let $\al\in(0,1)$ and $t\in\N$ be fixed such that $t > \left(\frac{1-\al}{2\al}\right)^2+1$. For any $n\in \N$ and any spherical $\sset{\pm\al}$-code $C_\al$ in $\R^n$, there exists a spherical $L(\al, t)$-code $C$ in $\R^n$ such that $\abs{C_\al} \le \abs{C} + O_{\al,t}(1)$.
\end{lemma}

\begin{proof}
  Denote by $R(\cdot, \cdot)$ the Ramsey number. Let $G$ be the underlying graph of $C_\al$. The clique number of the underlying graph $G$ is $\le 1+\floor{1/\al}$. In fact, if $K$ is a clique in $G$, then
  \[
    0 \le \normsq{\sum_{v\in K}v} = \sum_{v\in K}\normsq{v} + \sum_{\substack{v_1, v_2\in K\\v_1\neq v_2}}\ip{v_1}{v_2} = \abs{K} - \abs{K}(\abs{K}-1)\al \implies \abs{K} \le 1 + 1/\al.
  \]

  If $\abs{C_\al} < R(2+\floor{1/\al}, t)$, then we are done by taking $C = \emptyset$. Otherwise, by Ramsey's theorem, $G$ contains an independent set $I$ of size $t$. For every $v \in C_\al\setminus I$, we may opt to switch $v$ to $-v$ to ensure that the degree of $v$ to $I$ is at most $t/2$. We partition $C_\al\setminus I$ by how the vertices are connected to $I$. In particular, each part in the partition is indexed by some $I_1 \subset I$ with $\abs{I_1}\le t/2$, and is defined by
  \[
    C_\al(I_1) := \dset{v\in C_\al\setminus I}{\forall u\in I\setminus I_1\ u\not\sim v \text{ and } \forall u\in I_1\ u\sim v}.
  \]

  For every $I_1 \subset I$ such that $I_1\neq \emptyset$ and $\abs{I_1} \le t/2$, we bound the independent number of the subgraph $G[C_\al(I_1)]$ induced on $C_\al(I_1)$ as follows. Suppose that $J\subseteq C_\al(I_1)$ is an independent set of $G[C_\al(I_1)]$. Let $I_0 := I \setminus I_1$ and set \[
    u_0 := \sum_{v\in I_0}v,\quad u_1 := \sum_{v\in I_1}v,\quad u_2 := \sum_{v\in J}v.
  \]
  One can work out the Gram matrix of $u_0, u_1, u_2$: \[
    \begin{pmatrix}
      \abs{I_0}+\abs{I_0}(\abs{I_0}-1)\al & \abs{I_0}\abs{I_1}\al & \abs{I_0}\abs{J}\al\\
      \abs{I_0}\abs{I_1}\al & \abs{I_1}+\abs{I_1}(\abs{I_1}-1)\al & -\abs{I_1}\abs{J}\al\\
      \abs{I_0}\abs{J}\al & -\abs{I_1}\abs{J}\al & \abs{J}+\abs{J}(\abs{J}-1)\al\\
    \end{pmatrix}.
  \] Because the Gram matrix is positive semidefinite, its determinant \[
    \abs{I_0}\abs{I_1}\abs{J}\left((1-\al)^2(1-\al+\al(\abs{I_0}+\abs{I_1}))-4\al^3\left(\abs{I_0}\abs{I_1}-\left(\frac{1-\al}{2\al}\right)^2\right)\abs{J}\right)
  \] is non-negative. Together with $\abs{I_0} + \abs{I_1} = \abs{I} =t$ and $\abs{I_0}\abs{I_1} \ge t-1 > \left(\frac{1-\al}{2\al}\right)^2$, we obtain a bound on the size of the independent set $J$: \[
    \abs{J} \le \frac{(1-\al)^2(1-\al+\al(\abs{I_0}+\abs{I_1}))}{4\al^3\left(\abs{I_0}\abs{I_1}-\left(\frac{1-\al}{2\al}\right)^2\right)} = O_{\al,t}(1).
  \]

  Again, since the clique number of $G[C_\al(I_1)]$ is $\le 1+\floor{1/\al}$, by Ramsey's theorem, $\abs{C_\al(I_1)} = O_{\al,t}(1)$. We thus have the estimation
  \[
    \abs{C_\al} = \abs{I} + \abs{C_\al(\emptyset)} + \sum_{\substack{\emptyset \neq I_1 \subset I\\\abs{I_1}\le t/2}}\abs{C_\al(I_1)} = \abs{C_\al(\emptyset)} + O_{\al,t}(1).
  \]

  Pick any $v \in C_\al(\emptyset)$ and consider the projection $p(v)$ of $v$ onto the orthogonal complement of $\mathrm{span}(I)$. The vector $p(v)$ is given by $v - \sum_{u\in I}c_u u$ for some $\bm{c} = \dset{c_u}{u\in I}$ that minimizes $m(\bm{c}) := \normsq{v-\sum_{u\in I}c_u u}$.
  Since $\sset{v} \cup I$ is an independent set, the quadratic function $m(\bm{c})$ is symmetric with respect to $\dset{c_u}{u\in I}$, and so it is minimized at $\bm{c} = c\bm{1}$ for some $c\in \R$.
  Note that
  \[
    m(c\bm{1}) = \normsq{v - c\sum_{u\in I}u} = \normsq{v} - 2c\sum_{u\in I}\ip{v}{u} + c^2\normsq{\sum_{u\in I}u} = 1 - 2ct\al + c^2(t + t(t-1)\al),
  \]
  which is minimized at
  \begin{equation} \label{def_c}
    c := \frac{\al}{1+(t-1)\al}.
  \end{equation}
  Therefore $p(v) = v - c\sum_{u\in I}u$.

  Denote by $\hat{p}(v) := p(v)/\norm{p(v)}$ the normalized vector of $p(v)$. Lastly we show that $$C := \dset{\hat{p}(v)}{v\in C_{\al}(\emptyset)}$$ is a spherical $L(\al, t)$-code. For any $v_1, v_2\in C_\al(\emptyset)$, we compute $\ip{p(v_1)}{p(v_2)}$ in terms of $\ip{v_1}{v_2}$:
  \begin{multline} \label{ippv}
    \ip{p(v_1)}{p(v_2)} = \ip{v_1 - c\sum_{u\in I}u}{v_2 - c\sum_{u\in I}u} = \ip{v_1}{v_2} - c\sum_{u\in I}\ip{v_1}{u} - c\sum_{u\in I}\ip{v_2}{u} + c^2\normsq{\sum_{u\in I}u} \\ = \ip{v_1}{v_2} - ct\al - ct\al + c^2(t+t(t-1)\al) \stackrel{\eqref{def_c}}{=} \ip{v_1}{v_2} - 2ct\al + ct\al = \ip{v_1}{v_2} - ct\al.
  \end{multline}
  For any $v_1, v_2 \in C_{\al}(\emptyset)$ such that $v_1 \neq v_2$, we get
  \[
    \ip{\hat{p}(v_1)}{\hat{p}(v_2)} = \frac{\ip{p(v_1)}{p(v_2)}}{\norm{p(v_1)}\norm{p(v_2)}} \stackrel{\eqref{ippv}}{=} \frac{\ip{v_1}{v_2}-ct\al}{\sqrt{\ip{v_1}{v_1}-ct\al}\sqrt{\ip{v_2}{v_2}-ct\al}} = \frac{\ip{v_1}{v_2}-ct\al}{1-ct\al},
  \]
  which takes its value in $\sset{\frac{\pm \al - ct\al}{1-ct\al}} \stackrel{\eqref{def_c}}{=} L(\al, t)$.
\end{proof}


The following lemma, which will be applied in conjunction with Lemma~\ref{balla}, imposes structural restriction on the underlying graph of a spherical $L(\al, t)$-code.

\begin{lemma} \label{balla_cor}
  Given $\al \in (0,1)$ and a finite family $\G$ of graphs with spectral radius $>\la := \frac{1-\al}{2\al}$, there exists $t \in \N$ such that the underlying graph of any spherical $L(\al, t)$-code does not contain any graph in $\G$ as a subgraph.
\end{lemma}

We recall a necessary condition on eigenvalues of the sum of two matrices.

\begin{theorem}[Weyl's inequality~\cite{MR1511670}]
  Given two $n\times n$ Hermitian matrices $A$ and $B$. Denote the eigenvalues of $A$ as $\la_1(A)\ge \la_2(A) \ge \dots \ge \la_n(A)$, and similarly denote the eigenvalues of $B$ and $A+B$. Whenever $0 \le i, j, i+j < n$, $\la_{i+j+1}(A+B) \le \la_{i+1}(A) + \la_{j+1}(B)$.
\end{theorem}

We shall denote by $I_n$ the identity matrix of order $n$ and $J_n$ the all-ones matrix of order $n$, and we suppress the subscripts when the order of the matrices could be inferred from the context.

\begin{proof}[Proof of Lemma~\ref{balla_cor}]
  We first choose $t\in\N$ depending only on $\al$ and $\G$. Because $\la_1(G_0) > \la$ for every $G_0\in\G$, we thus choose $t \in \N$ large enough so that
  \begin{equation} \label{choose_t}
    1 - \frac{\la_1(G_0)}{\la} + \frac{v(G_0)}{t + \al^{-1} - 1} < 0
    \text{ for all } G_0\in\G,
  \end{equation}
  where $v(G_0)$ is the order in $G_0$.

  We claim that the underlying graph $G$ of any spherical $L(\al, t)$-code does not contain any graph in $\G$ as a subgraph. Suppose on the contrary that $G$ contains some graph $G_0\in\G$. Let $G_0'$ be the minimal induced subgraph of $G$ that contains $G_0$ as a subgraph, and $C_0'$ the vertex set of $G_0'$. Note that $v(G_0') = v(G_0)$. By the monotonicity of $\la_1$ and the choice of $t$ in \eqref{choose_t}, we obtain
  \begin{equation} \label{g1p}
    1 - \frac{\la_1(G_0')}{\la} + \frac{v(G_0')}{t + \al^{-1} - 1} < 0.
  \end{equation}
  Let $M_0'$ be the Gram matrix of $C_0'$ and let $A_0'$ be the adjacency matrix of $G_0'$. Since $C_0'$ is still a spherical $L(\al, t)$-code, the two matrices $M_0'$ and $A_0'$ are related by the equation \[
    \left(1+\frac{1}{t + \al^{-1}-1}\right)M_0' = I - \frac{A_0'}{\la} + \frac{J}{t + \al^{-1} - 1}.
  \]
  Using the fact that $\la_1(J) = v(G_0')$ and \eqref{g1p}, we know from Weyl's inequality that the least eigenvalue of $M_0'$ is negative. This contradicts with the fact that a Gram matrix is positive semidefinite.
\end{proof}

\subsection{Application of Theorem~\ref{forb_main}} \label{3.2}

Our application of Theorem~\ref{forb_main} addresses the upper bound on $N_\al(n)$ in Theorem~\ref{equi_main} when $\F(\la)$ has a finite forbidden subgraphs characterization.

\begin{corollary} \label{forb_main_cor}
  Define $\la^*, \al_2, \al_3, \dots$ as in Theorem~\ref{forb_main}. Given $\al \in (0,1)$ such that $\la := \frac{1-\al}{2\al} < \la^*$ and $\la \not\in \sset{\al_2, \al_3, \dots}$, the maximum number $N_\al(n)$ of equiangular lines in $\R^n$ with angle $\arccos\al$ is at most $\frac{k(\la)}{k(\la)-1}\cdot n + O(1)$, where $k(\la)$ is the spectral radius order of $\la$; in case $k(\la) = \infty$, $N_\al(n) \le n+O(1)$.
\end{corollary}

We recall a classical fact about the spectral radius of a connected graph and prove an immediate consequence.

\begin{theorem}[Corollary of the Perron--Frobenius theorem \cite{zbMATH02624926,MR1511438}]
  For every connected graph $G$, the largest eigenvalue $\la_1(G)$ has multiplicity $1$, with an eigenvector whose components are all positive.
\end{theorem}

\begin{corollary} \label{pf_cor}
  For every $\la > 0$ and every connected graph $G$ with spectral radius $\le \la$, $$v(G) \le \frac{k(\la)}{k(\la)-1}\cdot \rank{I - \frac{A}{\la}},$$ where $v(G)$ is the order of $G$, $k(\la)$ is the spectral radius order of $\la$, and $A$ is the adjacency matrix of $G$; in case $k(\la) = \infty$, $v(G) \le \rank{I-A/\la}$.
\end{corollary}

\begin{proof}
  If $\la_1(G) = \la$, the definition of $k(\la)$ implies that $v(G) \ge k(\la)$, and so the Perron--Frobenius theorem implies that $\rank{I - A/\la} = v(G) - 1 \ge (1 - 1/k(\la))v(G)$; otherwise $\la_1(G) \neq \la$, or equivalently $\la_1(G) < \la$, and so $\rank{I - A/\la} = v(G)$. Therefore when $k(\la) < \infty$ we always have $\rank{I - A/\la} \ge (1 - 1/k(\la))v(G)$; in case $k(\la) = \infty$, the definition of $k(\la)$ implies that $\la_1(G)\neq \la$, and so $\rank{I - A/\la} \ge v(G)$.
\end{proof}

The last preparation for the proof of Corollary~\ref{forb_main_cor} links $\rank{I - A / \la}$ with the dimension $n$.

\begin{proposition} \label{rank}
  Given $\al \in (0,1)$, $t\in \N$, for any spherical $L(\al, t)$-code $C$ in $\R^n$, the adjacency matrix $A$ of the underlying graph of $C$ satisfies $\rank{I - {A}/{\la}} \le n + 1$, where $\la := \frac{1-\al}{2\al}$.
\end{proposition}

\begin{proof}
  Let $M$ be the Gram matrix of $C$. Since $C$ is a spherical $L(\al, t)$-code, the matrices $M$ and $A$ are related by
  \[
    \left(1+\frac{1}{t+\al^{-1}-1}\right)M = I - \frac{A}{\la} + \frac{J}{t+\al^{-1}-1}.
  \]
  Therefore $\rank{I - A/\la} \le \rank{M} + \rank{J} \le n + 1$.
\end{proof}

\begin{proof}[Proof of Corollary~\ref{forb_main_cor}]
  Since $\la = \frac{1-\al}{2\al} \in (0, \la^*) \setminus \sset{\al_2, \al_3, \dots}$, by Theorem~\ref{forb_main}, there exists a finite family $\G$ of graphs such that $\F(\lambda)$ consists exactly of the connected graphs which do not contain any graph in $\G$ as a subgraph.

  Suppose $C_\al$ is a spherical $\sset{\pm\al}$-code in $\R^n$. By Lemma~\ref{balla} and Lemma~\ref{balla_cor}, we can choose $t\in \N$ depending only on $\al$ such that there exists a spherical $L(\al, t)$-code $C$ in $\R^n$ with $\abs{C_\al} \le \abs{C} + O(1)$ whose underlying graph $G$ does not contain any graph in $\G$.

  Let $G_1, \dots, G_m$ be the connected components of $G$. By the claim that $G$ does not contain any graph in $\G$, neither does $G_i$ for every $i\in[m]$. Because $\G$ is a forbidden subgraphs characterization for $\F(\la)$, the connected graph $G_i\in \F(\la)$, or in other words, $\la_1(G_i)\le\la$. Let $v(G_i)$ be the order of $G_i$ and $A_i$ the adjacency matrix of $G_i$. By Corollary~\ref{pf_cor} and Proposition~\ref{rank}, we have \[
    \abs{C} = \sum_{i=1}^m v(G_i) \le \frac{k(\la)}{k(\la)-1} \sum_{i=1}^m\rank{I - \frac{A_i}{\la}} = \frac{k(\la)}{k(\la)-1}\cdot \rank{I - \frac{A}{\la}} \le \frac{k(\la)}{k(\la)-1}\cdot (n+1).
  \]
  In case $k(\la) = \infty$, by Corollary~\ref{pf_cor} and Proposition~\ref{rank}, we have \[
    \abs{C} = \sum_{i=1}^m v(G_i) \le \sum_{i=1}^m\rank{I - \frac{A_i}{\la}} = \rank{I - \frac{A}{\la}} \le n+1. \qedhere
  \]
\end{proof}

\subsection{Proof of Theorem~\ref{equi_main}} \label{3.3}

The following lemma connects a spherical $\sset{\pm\al}$-code with a spherical $\sset{-1/\la, 0}$-code, and it provides the lower bound on $N_\al(n)$ in Theorem~\ref{equi_main}.

\begin{lemma} \label{lower_bound}
  Given $\al\in(0,1)$, let $\la := \frac{1-\al}{2\al}$. For any spherical $\sset{-1/\la, 0}$-code $C_0$ in $\R^k$, there exists a spherical $\sset{\pm\al}$-code in $\R^n$ of size $\lfloor\frac{n-1}{k}\rfloor\abs{C_0}$. In particular, the maximum number $N_\al(n)$ of equiangular lines in $\R^n$ with angle $\arccos\al$ is at least $\lfloor\frac{n-1}{k(\la)-1}\rfloor k(\la)$, where $k(\la)$ is the spectral radius order of $\la$; in case $k(\la) = \infty$, we have $N_\al(n) \ge n$.
\end{lemma}

\begin{proof}
  Given a spherical $\sset{-1/\la, 0}$-code $C_0$ in $\R^k$. Let $m := \lfloor\frac{n-1}{k}\rfloor$ and $M_{0}$ be the Gram matrix of $C_0$. Consider the matrix $M := (1-\al) M_0 \otimes I_m + \al J$ of order $m\abs{C_0}$. For both $M_0$ and $J$ are positive semidefinite, so is $M$. The rank of $M$ is at most $m\cdot \rank{M_0} + 1 \le mk+1 \le n$. Moreover, the diagonal entries of $M$ are ones and its off-diagonal entries are either $-\al$ or $\al$. Therefore $M$ can be realized as the Gram matrix of a spherical $\sset{\pm\al}$-code of size $m\abs{C_0}$ in $\R^n$.

  It suffices to construct a spherical $\sset{-1/\la, 0}$-code $C_0$ in $\R^{k(\la)-1}$ of size $k(\la)$. Because $\la$ is the spectral radius of a graph $G$ on $k(\la)$ vertices, $I - A/\la$ is positive semidefinite, where $A$ is the adjacency matrix of $G$, and $\rank{I - A/\la} \le k(\la)-1$. Clearly, $I-A/\la$ can be realized as the Gram matrix of a spherical $\sset{-1/\la, 0}$-code of size $k(\la)$ in $\R^{k(\la)-1}$.

  In case $k(\la) = \infty$, $(1-\al)I_n + \al J_n$ can be realized as the Gram matrix of a spherical $\sset{\pm\al}$-code of size $n$ in $\R^n$.
\end{proof}

In view of Corollary~\ref{forb_main_cor} and Lemma~\ref{lower_bound}, we are left to prove Theorem~\ref{equi_main} for the exceptional $\la \in \sset{\al_2,\al_3,\dots} \cup \sset{\la^*}$. Clearly, if $\la$ is an eigenvalue of the adjacency matrix of a graph, then
\begin{enumerate}[nosep]
  \item $\la$ is an algebraic integer --- it is a root of some monic polynomial with coefficients in $\Z$,
  \item $\la$ is totally real --- its conjugate elements are in $\R$.
\end{enumerate}
On the converse, it follows from Estes~\cite[Theorem~1]{MR1189507} that any totally real algebraic integer occurs as an eigenvalue of the adjacency matrix of a graph.

When $\la$ is not a totally real algebraic integer, the spectral radius order $k(\la) = \infty$, and Conjecture~\ref{main_conj} predicts that $N_\al(n) = n + O(1)$, which we prove in the affirmative.

\begin{proposition} \label{equi_not_real}
  Given $\al\in(0,1)$, let $\la := \frac{1-\al}{2\al}$. If $\la$ is not a totally real algebraic integer, then $n \le N_\al(n) \le n+1$.
\end{proposition}

\begin{proof}
  Let $C$ be a spherical $\sset{\pm\al}$-code in $\R^n$. Let $M$ be its Gram matrix, and $A$ the adjacency matrix of the underlying graph $G$. We know that $M = (1-\al)(I - A/\la) + \al J$. Since $\la$ is not an eigenvalue of $A$, $\rank{I-A/\la} = \abs{C}$ and so $n \ge \rank{M} \ge \rank{I - A/\la} - \rank{J} = \abs{C}-1$. Together with Lemma~\ref{lower_bound}, we have $n \le N_\al(n) \le n+1$.
\end{proof}

Observe that one of the conjugate elements $\sqrt{2-\sqrt{5}}$ of $\la^* = \sqrt{2 + \sqrt{5}}$ is not real. Proposition~\ref{equi_not_real} readily gives the proof of Theorem~\ref{equi_main} for $\la = \la^*$. The rest of the cases $\la \in \sset{\al_2, \al_3, \dots}$ in Theorem~\ref{equi_main} follow similarly from the following proposition, the proof of which is due in Appendix~\ref{prop_not_real}.

\begin{proposition} \label{not_real}
  The algebraic integers $\al_2, \al_3, \dots$, defined as in Theorem~\ref{forb_main}, are not totally real.
\end{proposition}

\subsection{An improved upper bound on $N_\al(n)$} \label{3.4}

In this section, we prove $N_\al(n) \le 1.49n+O(1)$ for every $\al \neq 1/3, 1/5, 1/(1+2\sqrt{2})$. We recall the following spectral tool bounding the rank of a symmetric matrix in terms of its trace and the trace of its square.

\begin{lemma} \label{balla1}
  The rank of a symmetric matrix $A$ is greater than or equal to ${(\tr{A})^2}/{\tr{A^2}}$.
\end{lemma}

This lemma appeared as Inequality 14 of Bellman~\cite[page 137]{MR1455129} in the form of an exercise. We refer the readers to \cite[Lemma~2.9]{MR3742757} for a short proof of Lemma~\ref{balla1}. Additionally, we develop a spectral result, which can be seen as an averaged version of the following lemma.

\begin{lemma}[Lemma~2.13 of Balla et al.~\cite{MR3742757}] \label{balla2}
  Let $G$ be a graph with minimum degree $\delta \ge 2$. Let $v_0$ be a vertex of $G$ and let $H$ be the subgraph consisting of all vertices within distance $k$ of $v_0$. Then $\la_1(H)\ge \frac{2k}{k+1}\sqrt{\delta-1}$.
\end{lemma}

\begin{lemma} \label{local}
  Let $G$ be a graph with average degree $d \ge 2$. There exists a vertex $v_0$ of $G$ such that $\la_1(H) \ge 2\cos(\tfrac{\pi}{k+2})\sqrt{d-1}$, where $H$ is the subgraph consisting of all vertices within distance $k$ of $v_0$.
\end{lemma}

By Lemma~\ref{abcdef}(p), the coefficient $2\cos(\tfrac{\pi}{k+2})$ in Lemma~\ref{local} equals $\la_1(P_{k+1})$, where $P_{k+1}$ is the path with $k$ edges. Notice that $\la_1(P_{k+1})$ is greater than the average degree $\frac{2k}{k+1}$ of $P_{k+1}$, which is the coefficient in Lemma~\ref{balla2}. This comparison between Lemma~\ref{balla2} and Lemma~\ref{local} is reminiscent of that between Nilli's~\cite{MR1124768} and Friedman's~\cite[Corollary~3.7]{MR1208809} (see also \cite{MR2056091}) proofs of the Alon--Boppana bound on the second largest eigenvalue of a regular graph. In fact, Lemma~\ref{local} was recently applied~\cite[Theorem~8]{MR3926279} to improve Hoory's bound~\cite[Theorem~3]{MR2102266} on the second largest eigenvalue for a class of graphs that are not necessarily regular.

\begin{proof}
  Since removing leaf vertices from a graph of average degree $d \ge 2$ cannot decrease its average degree, without loss of generality, we may assume that the minimum degree of $G$ is $\ge 2$. A walk $(v_{-1}, v_0, v_1, \dots)$ on $G$ is non-backtracking if $v_i \neq v_{i+2}$ for all $i$. For all $i \ge 0$, define $W_i$ to be the set of all non-backtracking walks $(v_{-1}, v_0, v_1, \dots, v_i)$ on $G$ of length $i + 1$. Define the forest $T$ as follows: the vertex set is $\cup_{i=0}^k W_i$ and two vertices are adjacent if and only if one is a simple extension of the other. For every $e = (v_{-1}, v_0) \in W_0$, denote by $T_{e}$ the connected component of $T$ containing $e$. We also denote by $G_{v_0}$ the subgraph of $G$ consisting of all vertices within distance $k$ of $v_0$.

  We claim that $\la_1(T_{e}) \le \la_1(G_{v_0})$ for every $e = (v_{-1}, v_0)\in W_0$. Let $s_i$ and $t_i$ be the number of closed walks of length $i$ starting respectively at $e$ in $T_{e}$ and $v_0$ in $G_{v_0}$. It is well known that $\la_1(T_{e}) = \limsup \sqrt[i]{s_i}$ and $\la_1(G_{v_0}) = \limsup \sqrt[i]{t_i}$. We naturally map a closed walk $e = e_0, e_1, \dots, e_i = e$ in $T_e$ to a closed walk $v_0, v_1, \dots, v_i = v_0$ in $G_{v_0}$, where $v_j$ is the terminal vertex of the non-backtracking walk $e_j$ for $j = 0, 1, \dots, i$. One can show that this map is injective, and so $s_i \le t_i$ for all $i$, from which the claim follows.

  Because $\la_1(T) = \max\dset{\la_1(T_{e})}{e\in W_0}$, it suffices to prove $\la_1(T) \ge \la\sqrt{d-1}$, where $\la := \la_1(P_{k+1}) = 2\cos(\tfrac{\pi}{k+2})$. Consider the non-backtracking random walk on $T$, where the start vertex $w_0 = (v_{-1}, v_0)$ is chosen uniformly at random from $W_0$ and, for $i\in[k]$, at $i$th step the next vertex $w_i = (v_{-1}, v_0, \dots, v_i)$ is chosen uniformly at random among the available choices in $W_i$. The transition matrix of this walk is
  \[
    P_{(v_{-1}, v_0, \dots, v_i), (v_{-1}, v_0, \dots, v_{i+1})} = \frac{1}{d(v_i)-1},
  \]
  where $d(v)$ denotes the degree of $v$ in $G$. Clearly $\abs{W_0} = d\abs{V(G)}$. Since $W_i$ is a finite set and $w_i = (v_{-1}, v_0, \dots, v_i)$ is a random element of $W_i$ with distribution \[
    p(w_i) := \frac{1}{d\abs{V(G)}}\prod_{j=0}^{i-1}\frac{1}{d(v_j)-1},
  \] for any $c\from W_i \to \R$, the basic identity of importance sampling allows us to represent $\sum_{w\in W_i}c(w)$ as follows:
  \begin{equation} \label{sample}
    \sum_{w\in W_i}c(w) = \sum_{w\in W_i}p(w)\frac{c(w)}{p(w)} = \E{{c(w_i)}/{p(w_i)}}.
  \end{equation}

  Let $(x_0, x_1, \dots, x_k) \in \R^{k+1}$ be an eigenvector of $P_{k+1}$ such that $x_0, x_1, \dots x_k > 0$. Define the vector $f\from V(T) \to \R$ by $f(w) = x_i\sqrt{p(w)}$ for $w\in W_i$, and define the matrix $A$ to be the adjacency matrix of the forest $T$. For $w = (u_{-1}, u_0, \dots, u_i)$, denote by $w^- = (u_{-1}, u_0, \dots, u_{i-1})$. By the importance sampling identity~\eqref{sample}, we observe that
  \begin{align*}
    \ip{f}{f} & = \sum_{i=0}^{k}\sum_{w\in W_i}f(w)^2 = \sum_{i=0}^k\E{x_i^2} = \sum_{i=0}^kx_i^2,\\
    \tfrac{1}{2}\ip{f}{Af} & = \sum_{i=1}^{k}\sum_{w\in W_i}f(w^-)f(w) = \sum_{i=1}^k\E{x_{i-1}x_i\sqrt{p(w^-)/p(w)}} = \sum_{i=1}^kx_{i-1}x_i\E{\sqrt{d(v_{i-1})-1}}.
  \end{align*}

  It can be verified by induction that $(v_{i-1}, v_i)$ is uniformly distributed on $W_0$ for all $i=0,1,\dots, k$. Thus $\pr{v_{i-1} = v} = \tfrac{d(v)}{d\abs{V(G)}} =: \pi(v)$ for all $v\in V(G)$ and $i = 1, 2, \dots, k$. Since each $v_{i-1}$ has distribution $\pi$ and the function $x\mapsto x\sqrt{x-1}$ is convex on $[2, \infty]$, Jensen's inequality gives \[
    \tfrac{1}{2}\ip{f}{Af} = \sum_{i=1}^k x_{i-1}x_i\sum_{v\in V(G)}\frac{d(v)}{d\abs{V(G)}}\sqrt{d(v)-1} \ge \sqrt{d-1}\sum_{i=1}^kx_{i-1}x_i.
  \]
  Finally we invoke the Rayleigh principle $2\sum_{i=1}^k{x_{i-1}x_i} = \la\sum_{i=0}^k x_i^2$ and $\la_1(T) \ge {\ip{f}{Af}}/{\ip{f}{f}}$.
\end{proof}

\begin{remark}
  Lemma~\ref{local} is asymptotically tight when $d$ is a prime plus one due to the existence of regular graphs of high girth. The Ramanujan graphs constructed independently by Margulis~\cite{MR671147} and Lubotzky, Phillips and Sarnak~\cite{MR963118} are $d$-regular graphs on $n$ vertices of girth $\Omega_d(\log n)$. In these Ramanujan graphs, every subgraph induced on the vertices within a bounded distance of a given vertex looks like a $d$-regular tree, whose spectral radius is bounded from above by the spectral radius $2\sqrt{d-1}$ of the infinite $d$-regular tree.
\end{remark}

\begin{remark}
  See the expository note by Levin and Peres~\cite{MR3681593} for other applications of Markov chains and importance sampling.
\end{remark}

\begin{theorem} \label{upper_bound}
  Given $\al\in(0,1)$, let $\la := \frac{1-\al}{2\al}$. If $\la \ge 2$, then the maximum number $N_\al(n)$ of equiangular lines in $\R^n$ is at most $\left(1+\frac{1}{4}+\frac{1}{\la^2}+o_\al(1)\right)n$. In particular, $N_{1/7}(n) \le \left(\frac{4}{3}+\frac{1}{36}+o(1)\right)n$.
\end{theorem}

\begin{proof}
  For a fixed $\eps > 0$, we shall prove that the size a spherical $\sset{\pm\al}$-code $C_\al$ in $\R^n$ is at most $\left(1+\frac{1}{4}+\frac{1}{\la^2} + \eps\right)n + O(1)$. Choose $k\in\N$ so that
  \begin{equation} \label{def_lap}
    \la' := 2\cos\left(\frac{\pi}{k+2}\right) > \frac{2}{\sqrt{1+4\eps}}.
  \end{equation}

  We first find a finite family $\G_0$ of graphs of spectral radius $>\la$ that approximates a forbidden subgraphs characterization of $\F(\la)$ in the following sense: if a connected graph $G$ does not contain any graph in $\G_0$, then either $G\in\F(\la)$ or the average degree of $G$ is at most
  \begin{equation} \label{def_d}
    d := \left(\frac{\la}{\la'}\right)^2 + 1.
  \end{equation}

  Choose $D\in\N$ such that the star $S_D\notin\F(\la)$ in view of Lemma~\ref{abcdef}(s). Suppose a connected graph $G$ does not contain $S_D$ and it has average degree $d_G > d$. Because $\la \ge 2$ hence $d_G > d \ge (2/\la')^2+1 > 2$, Lemma~\ref{local} implies that there exists a vertex $v_0$ of $G$ such that $\la_1(H) \ge \la'\sqrt{d_G-1} > \la'\sqrt{d-1} = \la$, where $H$ is the subgraph consisting of all vertices within distance $k$ of $v_0$. This means that $G$ contains a subgraph $H\notin \F(\la)$ with radius $\le k$. Therefore we can approximate a forbidden subgraphs characterization by
  \[
    \G_0 := \sset{S_D}\cup\sset{\text{connected graph }H\notin\F(\la)\text{ with maximum degree}< D\text{ and radius}\le k}.
  \]

  By Lemma~\ref{balla} and Lemma~\ref{balla_cor}, we can choose $t \in \N$ depending only on $\al$ such that there exists a spherical $L(\al, t)$-code $C$ in $\R^n$ with $\abs{C_\al}\le\abs{C}+O(1)$ whose underlying graph $G$ does not contain any graph in $\G_0$. Let $G_1, \dots, G_m$ be the connected components of $G$. As $G$ does not contain any graph of $\G_0$, neither does $G_i$ for all $i\in[m]$. By our choice of $\G_0$, we know that either $\la_1(G_i) \le \la$ or the average degree of $G_i$ is $\le d$.

  In the former case that $\la_1(G_i) \le \lambda$, Corollary~\ref{pf_cor} gives \[
    v(G_i) \le \frac{k(\la)}{k(\la)-1} \cdot \rank{I - \frac{A_i}{\la}},
  \]
  where $v(G_i)$ is the order of $G_i$, $k(\la)$ is the spectral radius order of $\la$, and $A_i$ is the adjacency matrix of $G_i$. Because the complete graph on $k(\la)$ vertices has the largest spectral radius $k(\la)-1$ among all graphs on $k(\la)$ vertices, we know that $\la$, the spectral radius of some graph on $k(\la)$ vertices, is at most $k(\la)-1$. Using the fact that $1/\la \le 1/4 + 1/\la^2$, we obtain \[
    v(G_i) \le \frac{k(\la)}{k(\la)-1}\cdot \rank{I - \frac{A_i}{\la}} \le \left(1+\frac{1}{\la}\right)\rank{I - \frac{A_i}{\la}} \le \left(1+\frac{1}{4}+\frac{1}{\la^2}\right)\rank{I - \frac{A_i}{\la}}.
  \] In case $k(\la) = \infty$, by Corollary~\ref{pf_cor}, $v(G_i) \le \rank{I - A_i/\la}$, so the same estimation holds trivially.

  In the latter case that the average degree $d(G_i)$ of $G_i$ is $\le d$, we can apply Lemma~\ref{balla1} to the matrix $I - A_i/\la$ and get
  \[
    \rank{I - \frac{A_i}{\la}} \ge \frac{(\tr{I-A_i/\la})^2}{\tr{(I-A_i/\la)^2}} = \frac{v(G_i)^2}{v(G_i) + v(G_i)d(G_i)/\la^2} \ge \frac{v(G_i)}{1 + d/\la^2}.
  \]
  Thus we obtain the estimation
  \begin{multline*}
    v(G_i)\le\left(1+\frac{d}{\la^2}\right)\rank{I - \frac{A_i}{\la}} \stackrel{\eqref{def_d}}{=} \left(1 + \frac{1}{\la'^2} + \frac{1}{\la^2}\right)\rank{I - \frac{A_i}{\la}} \\ \stackrel{\eqref{def_lap}}{\le} \left(1 + \frac{1}{4} + \frac{1}{\la^2} + \eps\right)\rank{I - \frac{A_i}{\la}}.
  \end{multline*}

  Summing up these estimations for all the $G_i$'s and using Proposition~\ref{rank}, we get
  \begin{multline*}
    \abs{C} = \sum_{i = 1}^m v(G_i) \le \left(1 + \frac{1}{4} + \frac{1}{\la^2} + \eps\right)\sum_{i=1}^m\rank{I-\frac{A_i}{\la}} \\
    \le \left(1 + \frac{1}{4} + \frac{1}{\la^2} + \eps\right)\rank{I-\frac{A}{\la}} \le \left(1 + \frac{1}{4} + \frac{1}{\la^2} + \eps\right)(n + 1).
  \end{multline*}
  Recalling that $\abs{C_\al} \le \abs{C} + O(1)$, we get $\abs{C_\al} \le \left(1+\frac{1}{4}+\frac{1}{\la^2}+\eps\right)n + O(1)$.
\end{proof}

\begin{corollary}
  Given $\al \in (0,1) \setminus \sset{1/3, 1/5, 1/(1+2\sqrt{2})}$, the maximum number $N_{\al}(n)$ of equiangular lines in $\R^n$ with angle $\arccos\al$ is at most $1.49n + O(1)$.
\end{corollary}

\begin{proof}
  Recall that $\la^* = \sqrt{2 + \sqrt{5}} \approx 2.058$. On the one hand, Theorem~\ref{equi_main} implies that $N_\al(n) \le \frac{4}{3}n+O(1)$ for $\la := \frac{1-\al}{2\al} \in (0, \la^*]\setminus\sset{1, \sqrt{2}, 2}$. On the other hand, because $\eps := 0.24 - 1/(\la^*)^2 > 0$, Theorem~\ref{upper_bound} implies that for $\la > \la^*$,
  $N_{\al}(n) \le \left(1+\frac{1}{4}+\frac{1}{\la^2}+\eps\right)n+O(1) \le 1.49n+O(1)$.
\end{proof}

\section{Concluding remarks} \label{open}

Besides Theorem~\ref{equi_main}, Lemma~\ref{lower_bound} and Proposition~\ref{equi_not_real}, we discuss two other evidences supporting Conjecture~\ref{main_conj}. Notice that the spectral radius order of $\la$ is at least the algebraic degree $\deg(\la)$ of $\la$. Conjecture~\ref{main_conj} predicts that $N_\al(n) \le \frac{\deg(\la)}{\deg(\la)-1}\cdot n + O(1)$. This is indeed a cheap upper bound on $N_\al(n)$.

\begin{proposition} \label{alg}
  Given $\al \in (0,1)$, if $\la := \frac{1-\al}{2\al}$ is a totally real algebraic integer, then $N_\al(n) \le \frac{\deg(\la)}{\deg(\la)-1}\cdot (n + 1)$, where $\deg(\la)$ is the algebraic degree of $\la$.
\end{proposition}

\begin{proof}
  Let $C$ be a spherical $\sset{\pm\al}$-code in $\R^n$. Let $M$ be its Gram matrix, and $A$ the adjacency matrix of the underlying graph $G$. We know that $M = (1-\al)(I - A/\la) + \al J$. If $\la$ is a totally real algebraic number, then the multiplicity of $\la$ as an eigenvalue of $A$ is $\le\frac{1}{\deg(\la)}\abs{C}$. Thus $\rank{I-A/\la} \ge \left(1-\frac{1}{\deg(\la)}\right)\abs{C}$ and so $n \ge \rank{M} \ge \rank{I - A/\la} - \rank{J} \ge \left(1-\frac{1}{\deg(\la)}\right)\abs{C} - 1$.
\end{proof}

Lemma~\ref{lower_bound} and Proposition~\ref{alg} would imply Conjecture~\ref{main_conj} in the equality case $k(\la) = \deg(\la)$. Note that $k(\la) = \deg(\la)$ if and only if $\la$ is the spectral radius of a graph with irreducible characteristic polynomial. A result of Mowshowitz~\cite{MR0337682} (see \cite[Theorem~3.8]{MR624545} for a generalization) states that a graph with irreducible characteristic polynomial has trivial automorphism group. Such graphs are known as asymmetric graphs. Erd\H{o}s and R\'enyi~\cite{MR0156334} showed that asymmetric graphs have at least $6$ vertices and there are $8$ asymmetric graphs on $6$ vertices. Interestingly, all these 8 graphs indeed have irreducible characteristic polynomial, and their spectral radii are $> \la^*$.

Clearly, if $\la$ is the spectral radius of a graph, then $\la$ is a totally real algebraic integer, and $\la$ is the largest among its conjugate elements. It would be interesting to study a complete set of necessary conditions for the spectral radius of a graph.

When $\la$ is a totally real algebraic integer but not the largest among its conjugate elements, the spectral radius order $k(\la) = \infty$, and Conjecture~\ref{main_conj} asserts that $N_\al(n) = n + O(1)$. This is indeed the case.

\begin{proposition}
  Given $\al \in (0,1)$, if $\lambda := \frac{1-\al}{2\al}$ is a totally real algebraic integer, but $\la$ is not the largest among its conjugate elements, then $n \le N_\al(n) \le n+2$.
\end{proposition}

\begin{proof}
  We denote by $\la_{-i}(\cdot)$ and $\la_i(\cdot)$ respectively the $i$th smallest eigenvalue and the $i$th largest eigenvalue of a matrix.
  Let $\la' > \la$ be a conjugate element of $\la$. Let $C$ be a spherical $\sset{\pm\al}$-code in $\R^n$. Let $M$ be its Gram matrix, and $A$ the adjacency matrix of the underlying graph $G$. We know that $M = (1-\al)(I - A/\la)+\al J$.

  Assume for the sake of contradiction that $\rank{I - A/\la} \le \abs{C}-2$, that is, $\la$ is an eigenvalue of $A$ with multiplicity $\ge 2$, then $1-\la'/\la < 0$ is an eigenvalue of $I - A/\la$ with multiplicity $\ge 2$, hence $\la_{-2}(I - A/\la) < 0$. By Weyl's inequality, $\la_{-1}(M) \le (1-\al)\la_{-2}(I-A/\la) + \al\la_2(J) < 0$.
  This contradicts with the fact that $M$ is positive semidefinite.

  Therefore $\rank{I - A/\la} \ge \abs{C}-1$ and so $n \ge \rank{M} \ge \rank{I - A/\la} - \rank{J} \ge \abs{C}-2$. Together with Lemma~\ref{lower_bound}, we have $n \le N_\al(n) \le n+2$.
\end{proof}

Lastly, we remark on a possible extension of our method. Our proof strategy would resolve Conjecture~\ref{main_conj} provided the following connection between forbidden subgraphs characterization and multiplicity of the second largest eigenvalue.

\begin{conjecture}
  For every $\lambda > 0$, there exist graphs $G_1, G_2, \dots, G_n$ of spectral radius $> \la$ such that for every connected graph $G$ that does not contain any of $G_1, G_2, \dots, G_n$ as a subgraph, if $\la$ is the second largest eigenvalue of $G$ then the multiplicity of $\la$ is $\le v(G)/k(\la)$, where $v(G)$ is the order of $G$ and $k(\la)$ is the spectral radius order of $\la$.
\end{conjecture}

In this direction, Woo and Neumaier~\cite{MR2365422} investigated the structure of graphs whose spectral radius is in $(2, 3/\sqrt{2}]$. In particular, such a graph is either an open quipu\footnote{An \emph{open quipu} is a tree of maximum degree $3$ such that all vertices of degree 3 lie on a path.}, a closed quipu\footnote{A \emph{closed quipu} is a connected graph of maximum degree $3$ such that all vertices of degree $3$ lie on a unique cycle.} or a dagger\footnote{A \emph{dagger} is $A_n$ defined in Lemma~\ref{abcdef}.}, for which we assert that the multiplicity of any eigenvalue larger than $2$ is at most $2$.

\section*{Acknowledgements}

Thanks to Boris Bukh for introducing equiangular lines to the first author, and to Jun Su and Sebastian Cioab\u{a} for useful correspondence. We wish to express our deep appreciation to the referee for meticulous reading, and for pointing out a mistake in Theorem~\ref{forb_main} and many other inaccuracies in an earlier version of the manuscript. All remaining errors are ours.

\bibliographystyle{alpha}
\bibliography{equiangular_lines}

\appendix

\section{Proof of Lemma~\ref{abcdef}} \label{app}

Given a connected graph $G$, let $A$ be its adjacency matrix. The Perron--Frobenius theorem asserts that the eigenvector $f\from V(G)\to \R$ corresponding to the unique largest eigenvalue of $A$ can be chosen so that all of its components are positive. As an eigenvector $f'$ corresponding to any other eigenvalue is orthogonal to $f$, $f'$ must have at least one negative component. Therefore to compute $\la_1(G)$ it suffices to demonstrate an eigenvector whose components are positive.

\begin{proof}[Proof of Lemma~\ref{abcdef}(b\textsubscript{1})]
  The eigenvector that maps the leaves to $1$ and the rest of the vertices to $2$ gives $\la_1(B_{1,n,1})=2$.
\end{proof}

\begin{proof}[Proof of Lemma~\ref{abcdef}(c)]
  The constant eigenvector that maps every vertex to $1$ gives $\la_1(C_n) = 2$.
\end{proof}

\begin{proof}[Proof of Lemma~\ref{abcdef}(p)]
  Let $v_1, \dots, v_n$ be the path $P_n$. The eigenvector that maps $v_i$ to $\sin(\frac{\pi i}{n+1})$ gives $\la_1(P_n) = 2\cos(\frac{\pi}{n+1})$.
\end{proof}

\begin{proof}[Proof of Lemma~\ref{abcdef}(s)]
  The eigenvector that maps the leaves to $1$ and the vertex of degree $n$ to $\sqrt{n}$ gives $\la_1(S_n) = \sqrt{n}$.
\end{proof}

For the proof of other facts in Lemma~\ref{abcdef}, we shall use the following lemmas due to Hoffman.

\begin{lemma}[Lemma~3.4 of Hoffman~\cite{MR0347860}] \label{lim0}
  Let $A_{-1}$ be a principal submatrix of order $n-1$ of a symmetric matrix $A_0$ of order $n$ with non-negative entries. Define $A_{i+1}$ recursively by \[
    A_{i+1} = \begin{pmatrix}
      A_i & e_i^T \\
      e_i & 0
    \end{pmatrix}, \text{ where }e_i = \begin{pmatrix}
      0 & 0 & \dots & 0 & 1
    \end{pmatrix}.
  \] Assume further that $\lim_{i\to\infty}\la_1(A_i) > 2$. Then $\lim_{i\to\infty}\la_1(A_i)$ is the largest positive root of
  \begin{equation} \label{charpoly}
    \left(\frac{x+\sqrt{x^2-4}}{2}\right)p_0(x) = p_{-1}(x),
  \end{equation}
  where $p_i$ is the characteristic polynomial of $A_i$ for $i=-1, 0$.
\end{lemma}

\begin{definition}
  Let $G$ be a connected graph, and let $v$ be a vertex of $G$. Denote $(G, v, n)$ the graph obtained from $G$ by appending a path of $n$ vertices to $G$ at $v$. Let $G_1, G_2$ be disjoint connected graphs, and let $v_1, v_2$ be vertices of $G_1, G_2$ respectively. Define $(G_1, v_1, n, v_2, G_2)$ to be the graph obtained from $G_1$ and $G_2$ by joining them by a path of $n$ vertices connecting $v_1$ and $v_2$.
\end{definition}

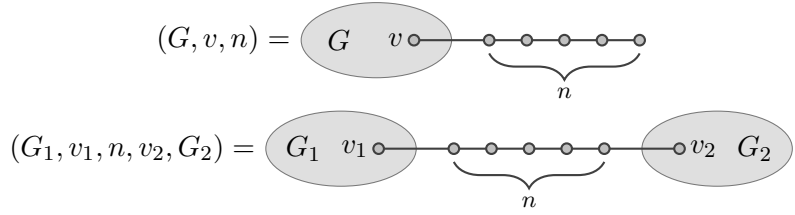
\begin{figure}[ht]
  \centering
  \begin{tikzpicture}[thick, scale=0.5, baseline=(v.base)]
    \draw[thin, fill=litegray, opacity=0.5] (-6,0) ellipse (2 and 1);
    \draw (-5,0) node[left]{$v$};
    \draw (-7,0) node{$G$};
    \draw[darkgray] (-5,0) node[vertex]{} -- (-3,0) node[vertex]{} -- (-2,0) node[vertex]{} -- (-1,0) node[vertex]{} -- (0,0) node[vertex]{} -- (1,0) node[vertex]{};
    \draw [braket] (1,0) -- (-3,0) node [black,midway,yshift=-15pt] {\footnotesize $n$};
    \draw (-8, 0) node[left]{$(G, v, n) = $};
  \end{tikzpicture}\qquad
  \begin{tikzpicture}[thick, scale=0.5, baseline=(v.base)]
    \draw[thin, fill=litegray, opacity=0.5] (-6,0) ellipse (2 and 1);
    \draw[thin, fill=litegray, opacity=0.5] (4,0) ellipse (2 and 1);
    \draw (-5,0) node[left]{$v_1$};
    \draw (-7,0) node{$G_1$};
    \draw (3,0) node[right]{$v_2$};
    \draw (5,0) node{$G_2$};
    \draw[darkgray] (-5,0) node[vertex]{} -- (-3,0) node[vertex]{} -- (-2,0) node[vertex]{} -- (-1,0) node[vertex]{} -- (0,0) node[vertex]{} -- (1,0) node[vertex]{} -- (3,0) node[vertex]{};
    \draw [braket] (1,0) -- (-3,0) node [black,midway,yshift=-15pt] {\footnotesize $n$};
    \draw (-8, 0) node[left]{$(G_1, v_1, n, v_2, G_2) = $};
  \end{tikzpicture}
\end{figure}

\begin{remark} \label{lim1}
  When we apply Lemma~\ref{lim0} to the adjacency matrix of a graph, we get the following interpretation. Let $G$ be a connected graph, and let $v$ be a vertex of $G$. Assume further that $\la_1(G, v, n) \ge 2$ for some $n$. Then $\lim\la_1(A_n)$ is the largest positive root of \eqref{charpoly}, where $p_{-1}, p_0$ are the characteristic polynomials of $G\setminus\sset{v}$ and $G$ respectively.
\end{remark}

\begin{lemma}[Proposition~4.2 of Hoffman~\cite{MR0347860}] \label{lim2}
  Let $G_1, G_2$ be disjoint connected graphs, $v_1, v_2$ vertices of degree $\ge 2$ of $G_1, G_2$ respectively. Then $$\lim_{n\to\infty}\la_1(G_1, v_1, n, v_2, G_2) = \max\sset{\lim_{n\to\infty}\la_1\left(G_1, v_1, n\right), \lim_{n\to\infty}\la_1\left(G_2, v_2, n\right)}.$$
\end{lemma}

\begin{definition}
  Let $e$ be an edge of a graph $G$. If there exists a path in $G$, $x_1, x_2, \dots, x_k$ where $x_{k-1}$ and $x_k$ are the end vertices of $e$, and the degrees of $x_1, x_2, \dots, x_{k-1}$ are respectively $1, 2, 2, \dots, 2$, then $e$ is said to be on an \emph{end path} of $G$.
\end{definition}

\begin{lemma}[Proposition~4.1 of Hoffman~\cite{MR0347860}] \label{mono}
  Let $G$ be a connected graph with $\la_1(G) > 2$, $e = (x, y)$ an edge of $G$ not on an end path of $G$. Let $G^+_e$ be the graph obtained from $G$ by deleting edge $e$, and adding a vertex $z$ adjacent to $x$ and $y$ only. Then $\la_1(G^+_e) < \la_1(G)$.
\end{lemma}

The monotonicity of the spectral radii of each family of graphs in Lemma~\ref{abcdef}(a,e,f) follows immediately from the monotonicity of the spectral radii. The monotonicity in Lemma~\ref{abcdef}(b\textsubscript{2},d) follows from Lemma~\ref{mono} and the facts that $\la_1(B_{m_1,0,m_2}) > \la_1(B_{1,0,1}) = 2$ for $(m_1,m_2)\neq (1,1)$ and $\la_1(D_2) > 2$.

We are left to compute the limits.

\begin{proof}[Proof of Lemma~\ref{abcdef}(a)]
  Note that $A_n = (S_3, v, n)$, where $v$ is the vertex of degree $3$ in $S_3$. Note that $\la_1(A_1) = 2$. By Remark~\ref{lim1}, $\lim \la_1(A_n)$ is the largest positive root of \[
    \left(\frac{x+\sqrt{x^2-4}}{2}\right)x^2(x^2-3) = x^3,
  \] which turns out to be $3/\sqrt{2}$.
\end{proof}

\begin{proof}[Proof of Lemma~\ref{abcdef}(f)]
  Note that $F_n = (P_5, v, n)$, where $v$ is the third vertex of $P_5$. Observe that $\la_1(F_2) = 2$. By Remark~\ref{lim1}, $\lim\la_1(F_n)$ is the largest positive root of
  \[
    \left(\frac{x+\sqrt{x^2-4}}{2}\right)\left(x^5 - 4x^3 + 3x\right) = \left(x^2-1\right)^2,
  \] which turns out to be $\la^*$.
\end{proof}

We need the characteristic polynomials of paths and cycles for Lemma~\ref{abcdef}(d, e). The readers are invited to derive them by reduction and induction.

\begin{lemma} \label{pc}
  Denote $p_n$ and $q_n$ the characteristic polynomials of $P_n$ and $C_n$ respectively. Then
  \begin{gather*}
    p_0(x) = 1,\quad p_1(x) = x, \quad p_{n} = xp_{n-1}(x) - p_{n-2}(x),\\
    q_{n+1}(x) = p_{n+1}(x) - p_{n-1}(x)-2, \quad\text{for all }n\ge 2.
  \end{gather*}
  Moreover, the recursions give
  \begin{equation} \label{rec}
    p_n(x) = \frac{\theta^n}{1-\theta^{-2}}+\frac{\theta^{-n}}{1-\theta^2}, \quad q_n(x) = \theta^n+\theta^{-n} - 2,
  \end{equation}
  where $\theta = \theta(x) := \frac{x+\sqrt{x^2-4}}{2}$.
\end{lemma}

\begin{proof}[Proof of Lemma~\ref{abcdef}(e)]
  For the $m = 1$ case, $\la_1(E_{1, n})$ is at least the average degree $2-2/(n+3)$ and so $\lim\la_1(E_{1,n})\ge 2$.  On the other hand, assume for the sake of contradiction that $\lim\la_1(E_{1,n}) > 2$. By Remark~\ref{lim1}, $\lim\la_1(E_{1,n})$ is the largest positive root of \[
    \left(\frac{x+\sqrt{x^2-4}}{2}\right)x\left(x^2 - 2\right) = x^2,
  \] which turns out to be $2$ contradicting the assumption $\lim\la_1(E_{1,n}) > 2$.

  For the $m \ge 2$ case, because $\la_1(E_{m, 8}) \ge \la_1(E_{2,8}) = 2$ and $E_{m,n} = (P_{m+2}, v, n)$, where $v$ is the second vertex of $P_{m+2}$, Remark~\ref{lim1} gives that $\lim_{n\to\infty}E_{m,n}$ is the largest positive root~of \[
    \theta p_{m+2}(x) = xp_m(x),
  \] where $\theta$ and $p_i$ are defined as in Lemma~\ref{pc}. From this and \eqref{rec}, using $x = \theta + 1/\theta$ and $z = \theta^2$, we seek the largest root $z_m$ of \[
    z^{m+1} = 1 + z + \dots + z^{m-1}.
  \] By the definitions of $\beta_m$ and $\al_m$ in Theorem~\ref{forb_main}, this proves that $z_m = \beta_m$ and $\lim_{n}\la_1(E_{m,n}) = \al_m$.
\end{proof}

\begin{proof}[Proof of Lemma~\ref{abcdef}(b\textsubscript{2})]
  Let $v_1$ and $v_2$ be the second vertices in $P_{m_1+2}$ and $P_{m_2+2}$ respectively. By Lemma~\ref{lim2},
  \begin{align*}
    \lim_{n\to\infty}\la_1(B_{m_1,n,m_2}) & = \max\sset{\lim_{n\to\infty}\la_1(P_{m_1+2}, v_1, n), \lim_{n\to\infty}\la_1(P_{m_2+2},v_2, n)} \\ & = \max\sset{\lim_{n\to\infty}\la_1(E_{m_1,n}), \lim_{n\to\infty}\la_1(E_{m_2,n})} = \max(\al_{m_1}, \al_{m_2}).\qedhere
  \end{align*}
\end{proof}

\begin{proof}[Proof of Lemma~\ref{abcdef}(d)]
  Let $M_n$ be the adjacency matrix of $D_n$ and let $r_n$ be its characteristic polynomial. By expanding the determinant of $xI - M_n$ along the row indexed by the leaf of $D_n$, one can obtain that $\la_1(D_n)$ is the largest root of $$r_n(x) = xq_{n+1}(x)-p_n(x),$$ where $p_{n}$ and $q_{n+1}$ are defined as in Lemma~\ref{pc}. From this and \eqref{rec}, using $x = \theta + 1/\theta$ and $z = \theta^2$, we seek the largest root $z_n$ of $$\left(z^2-z-1\right)\left(1-z^{-(n+1)}\right)=2z^{-n/2}\left(z^{1/2}+z^{-1/2}\right).$$ As $z_n > 1$, we get $z_n^2-z_n-1>0$, hence $z_n > \phi$, where $\phi = \frac{1+\sqrt{5}}{2}$ is the golden ratio. As $n\to\infty$, the largest root $z_n$ tends to $\phi$, and so $\lim\la_1(D_n) = \phi^{1/2}+\phi^{-1/2} = \la^*$.
\end{proof}

\section{Proof of Proposition~\ref{not_real}} \label{prop_not_real}

Fix $m \ge 2$ and let $\beta$ be the largest root of the equation $x^{m+1} = 1 + x + \dots + x^{m-1}$. We shall prove that $\al := \beta^{1/2} + \beta^{-1/2}$ is an algebraic integer but not totally real. We need the following lemma to characterize the conjugate elements of $\al$.

\begin{lemma} \label{rational_trans}
  Suppose that the polynomial $p(x) \in \Q[x]$ of degree $n$ has $n$ complex roots, say $\gamma_1, \gamma_2, \dots, \gamma_n$, counted with multiplicity. If these roots are not the poles of a rational function $r(x) \in \Q(x)$, then $q(x) := \prod_{i=1}^n(x-r(\gamma_i))$ is a polynomial in $\Q[x]$.
\end{lemma}

\begin{proof}
  For every $d \in [n]$, let $e_d\in \Z[x_1, x_2, \dots, x_n]$ be the elementary symmetric polynomial of degree $d$ in $n$ variables. Note that $f_d\left(x_1, x_2, \dots, x_n\right) := e_d\left(r(x_1), r(x_2), \dots, r(x_n)\right)$ is a symmetric rational function with rational coefficients. Thus by the fundamental theorem of symmetric functions, $f_d$ can be written as a rational function of the elementary symmetric functions $e_1, e_2, \dots, e_n$ with rational coefficients. Note that the coefficient of $x^d$ in the polynomial $q(x)$ is precisely $(-1)^{n-d}f_{n-d}(\gamma_1, \gamma_2, \dots, \gamma_n)$. Since $p(x)\in \Q[x]$, Vieta's formulas tell us $e_d(\gamma_1, \gamma_2, \dots, \gamma_n)$ is a rational for all $d\in[n]$, hence the coefficients of $q(x)$ are rational.
\end{proof}

\begin{proof}[Proof of Proposition~\ref{not_real}]
  We first prove the existence and the uniqueness of the positive root of $p(x) := x^{m+1} - (1 + x + \dots + x^{m-1})$. Since $p(1) = 1 - m < 0$ and $p(\infty) = \infty$, by the intermediate value theorem, $p$ has a positive root. Suppose for a moment that $\beta$ is just a positive root of $p$. Then $\beta^{m+1} = 1 + \beta + \dots + \beta^{m-1} > 1$, and so $\beta > 1$, hence $\beta^{m+1} = 1 + \beta + \dots + \beta^{m-1} > m$. Notice that
  \begin{equation} \label{beta_simple}
    \beta^{m+1} = 1 + \beta + \dots + \beta^{m-1} = \frac{\beta^m - 1}{\beta - 1} \implies \beta - 1 = \frac{1}{\beta} - \frac{1}{\beta^{m+1}}.
  \end{equation}
  This means that $\beta$ is a zero of the function $f(x) := x - 1 - \frac{1}{x} + \frac{1}{x^{m+1}}$ in $\left(m^{1/(m+1)},\infty\right)$. Note that $f'(x) = 1 + \frac{1}{x^2} - \frac{m+1}{x^{m+2}}$, and when $x > m^{1/(m+1)}$, the derivative $$f'(x) \ge \frac{2}{x} - \frac{m+1}{x^{m+2}} = \frac{2x^{m+1}-(m+1)}{x^{m+2}} \ge \frac{2m - (m+1)}{x^{m+2}} > 0,$$ which shows the uniqueness of the positive root of $p$.

  Next we show that $\al$ is an algebraic integer. Let $\beta = \beta_0, \beta_1, \dots, \beta_m$ be the $m+1$ complex roots of $p$ counted with multiplicity. Let $\gamma_i := \sqrt{\beta_i}$ be the principal square root of $\beta_i$ for $i \in \sset{0, 1, \dots, m}$. Clearly, $\pm\gamma_0, \pm\gamma_1, \dots, \pm\gamma_m$ are the roots of the monic polynomial $p(x^2) \in \Z[x]$, and $\pm 1/\gamma_0, \pm1/\gamma_1, \dots, \pm 1/\gamma_m$ are the roots of the monic polynomial $-x^{2(m+1)}p(1/x^2) \in \Z[x]$. Therefore both $\gamma_0$ and $1/\gamma_0$ are algebraic integers, and so is $\al = \gamma_0 + 1/\gamma_0$.

  Suppose $\al'$ is a conjugate element of $\al$. Applying Lemma~\ref{rational_trans} to $\pm\gamma_0, \dots, \pm\gamma_m$ and the rational function $x\mapsto x + 1/x$, we know that $\al'$ must be of the form $\pm(\gamma_i + 1/\gamma_i)$, for some $i \in \sset{0,1,\dots,m}$. Assume for the sake of contradiction that $\al' = \pm(\gamma_i + 1/\gamma_i)$ is real for some $i \neq 0$. Without loss of generality, we may assume that $i = 1$. We can solve the quadratic equation $\gamma_1^2 \mp \al'\gamma_1 + 1 = 0$ and get $\gamma_1 = \frac{\pm\al' \pm \sqrt{\al'^2-4}}{2}$, where the plus-minus signs are independent. As $\beta = \beta_0$ is the only positive root of $p$, $\gamma_1 = \sqrt{\beta_1}$ is not real. Since $\al' \in \R$ but $\gamma_1 \not\in \R$, it must be the case that $-2 < \al' < 2$ and $\gamma_1 = \frac{\pm\al' \pm i\sqrt{4-\al'^2}}{2}$. Therefore we have $$\abs{\beta_1} = \abs{\gamma_1}^2 = \left(\frac{\al'}{2}\right)^2 + \left(\frac{\sqrt{4-\al'^2}}{2}\right)^2 = 1.$$ Let $\theta$ be the argument of $\beta_1$, that is $\beta_1 = \cos\theta + i\sin\theta$. Since $\beta_1$ is a root of $p$, \eqref{beta_simple} holds for $\beta_1$, that is, $\beta_1 - 1 = \beta_1^{-1} - \beta_1^{-(m+1)}$, or $\beta_1^{-(m+1)} = 1 - \left(\beta_1 - \beta_1^{-1}\right) = 1 - 2i\sin\theta$ after rearrangement. Using $\abs{\beta_1} = 1$, we get $\sin\theta = 0$, hence $\beta_1 = \pm 1$. As $p(1) < 0$, $\beta_1 = -1$ is the only possibility, in which case $\al' = \pm(\sqrt{\beta_1}+1/\sqrt{\beta_1}) = 0$ and it cannot be a conjugate element of $\al$.

  The contradiction shows that any real conjugate element of $\al$ is of the form $\pm(\gamma_0+1/\gamma_0)$, that is, $\pm\al$. If $\al$ were a totally algebraic integer, its degree would be either $1$ or $2$. In the former case, $\al$ would be an integer, which contradicts with $2 < \al < \la^* < 3$. In the latter case, $\al^2$ would be an integer, which contradicts with $2 < \al < \la^* < \sqrt{5}$.
\end{proof}

\end{document}